\documentclass{amsart}
\usepackage{graphicx}
\usepackage{amsmath,amsfonts,amssymb,enumerate}

\newtheorem{thm}{Theorem}[section]
\newtheorem{cor}[thm]{Corollary}
\newtheorem{lem}[thm]{Lemma}
\newtheorem{prop}[thm]{Proposition}
\newtheorem{rem}[thm]{Remark}
\newtheorem{ex}[thm]{Example}
\numberwithin{equation}{section}

\newcommand{\F}{\mathbb{F}}
\renewcommand{\P}{\mathbb{P}}
\newcommand{\m}{\mathcal}
\begin{document}


%
%

\title{NEW EXAMPLES OF ASYMPTOTICALLY GOOD KUMMER TYPE TOWERS}

\author{MAR\'IA CHARA}

\address{ Instituto de Matemática Aplicada del Litoral (UNL-CONICET)\\ G\"{u}emes 3450 S3000GLN, Santa Fe, ARGENTINA \\}
\email{mchara@santafe-conicet.gov.ar}

\author{RICARDO TOLEDANO}

\address{Departamento de Matem\'atica, Facultad de Ingeniería Qu\'imica (UNL)\\
Santiago del Estero 2829, S3000AOM, Santa Fe, ARGENTINA\\}
\email{toledano@santafe-conicet.gov.ar}

\maketitle

\begin{abstract}
In this work, we give sufficient conditions in order to have finite ramification locus in sequences of function fields defined by different kind of Kummer extensions. These conditions can be easily implemented in a computer to generate several examples. We present some new examples of asymptotically good towers of Kummer type and we show that many known examples can be obtained from our general results.
\end{abstract}

\noindent \\

\section{Introduction}

Asymptotically good towers of function fields have received much attention in theoretical considerations related to coding theory and Cryptography after the work of Tsfasman, Vladut and Zink in \cite{TVZ82}. They showed the existence of linear codes with parameters improving the so-called Gilbert-Varshamov bound using asymptotically good towers of modular curves (in fact, optimal) and a construction of linear codes due to Goppa. However, they did not give a method for constructing them.
This motivated the search for asymptotically good towers of function fields over finite fields defined in an explicit way.
  It turns out that it is a non-trivial problem to provide examples of such towers. This line of research was initiated by Garcia and Stichtenoth. They established all the fundamentals results of the theory of asymptotically good towers, (see, for example \cite{GS07}), and made one of their most important contributions with the study of the so-called recursive towers (see Section~\ref{section2} for details).

The aim of this paper is to continue the investigation (initiated in \cite{Chara2011}) of the asymptotic behavior of towers of function fields defined by a Kummer equation of the form \begin{equation}\label{ecuintro}
  y^m=\frac{x^m-\alpha f(x)+\alpha}{f(x)},
\end{equation} where $f\in \F_q[x]$ is a suitable polynomial and $\alpha \in \F_q^*$. More precisely, in \cite{Chara2011} we obtained conditions in order to have non empty splitting locus of towers recursively defined by \eqref{ecuintro}, and now we will deal with the ramification locus. The finiteness of the ramification locus of towers recursively defined by \eqref{ecuintro} will suffice to prove their good asymptotic behavior because an important and well-known result of Garcia and Stichtenoth states that if a tame tower has non-empty splitting locus and finite ramification locus, then the tower is asymptotically good, \cite[Theorem 2.1]{GST97}. We will show new examples of asymptotically good towers recursively defined by \eqref{ecuintro}. 

In Section $2$ we give the basic definitions and we establish the notation to be used throughout the paper. In Section $3$ we prove our main results. In particular, in Theorem~\ref{teofiniteram}, we give sufficient conditions to have asymptotically good Kummer type towers recursively defined by \eqref{ecuintro}. The first part of Section $3$ is devoted to prove some auxiliary results needed in the proof of Theorem~\ref{teofiniteram}. An interesting feature of this results is that they can be easily implemented in a computer so that we were able to search for many equations of the form \eqref{ecuintro} defining good towers. Consequently, we show different examples of asymptotically good towers of Kummer type and we observe that some known examples can be obtained from our general results. In particular, in Example~\ref{ejem81} we present new interesting examples of asymptotically good Kummer type towers whose defining equation have coefficients in $\F_{9}\setminus \F_3$.


 \section{Preliminaries}\label{section2}
Let $q$ be a prime power. An algebraic function field $F/\F_q$ is a finite algebraic extension  of the rational function field $\F_q(x)$, where $x$ is a transcendental element over $\F_q$.

Let $\m{F}=(F_0, F_1, \ldots)$ be a sequence  of function fields over $\F_q$. We shall say that $\m{F}$ is \emph{admissible} if\begin{enumerate}[1.]
  \item $F_0\subsetneq F_1 \subsetneq F_2 \subsetneq \cdots$,
  \item the field extension $F_{i+1}/F_i$ is finite and separable for all $i\geq 0$, and
  \item the field $\F_q$ is algebraically closed in $F_i$ for all $i\geq 0$, i.e., the only elements of $F_i$ which are algebraic over $\F_q$ are the elements of $\F_q$. In this case we shall say that $\F_q$ is the full constant field of each $F_i$.
\end{enumerate}
If the genus $g(F_i)$ grows to infinity as $i\rightarrow \infty$, we say that the admissible sequence $\m{F}$ is a \emph{tower of function fields over} $\F_q$.

We shall say that an admissible sequence $\m{F}$ is \emph{recursively defined} if there exist a bivariate polynomial $H\in \F_q[S,T]$ and transcendental elements $x_i$, such that for all $i\geq0$ the following holds:
\begin{enumerate}[1.]
\item $F_0=\F_q(x_0)$ is the rational function field.
\item $F_{i+1}=F_i(x_{i+1})$ with $H(x_i, x_{i+1})=0$.
\item $[F_{i+1}:F_i]=\deg_{T}H$.
\end{enumerate}


Let $\P(F)$ denote the set of all places of a function field $F/\F_q$. The following definitions are important in the study of the asymptotic behavior of sequences of function fields. Let $\m{F}=(F_0, F_1, \ldots)$ be an admissible sequence  of function fields over $\F_q$. A place $P\in \mathbb{P}(F_i)$ {\em splits completely} in $\mathcal{F}$ if $P$ splits completely in each extension $F_j/F_i$. The {\em splitting locus} of $\mathcal{F}$ over $F_0$ is defined as
\[Split(\mathcal{F}/F_0):=\{P\in \mathbb{P}(F_0)\,:\,\text{$\deg(P)=1$ and $P$ splits completely in $\mathcal{F}$}\}\,,\] where $\deg(P)$ is the degree of the place $P$.
A place $P\in \mathbb{P}(F_i)$ is {\em ramified} in $\mathcal{F}$ if $P$ is ramified in any extension $F_j/F_i$.
The {\em ramification locus} of $\m{F}$ over $F_0$ is the set
$$Ram(\m{F}/F_0):=\{P\in \mathbb{P}(F_0)\,:\,\text{$P$ ramified in some extension }F_n/F_0\}.$$ A place $P\in \mathbb{P}(F_i)$ {\em is totally ramified} in $\mathcal{F}$ if $P$ is totally ramified in each extension $F_j/F_i$.  The {\em complete ramification locus} of $\mathcal{F}$ over $F_0$ is defined as
\[Cram(\mathcal{F}/F_0) :=\{P\in \mathbb{P}(F_0)\,:\,\text{$\deg(P)=1$ and $P$ is totally ramified in $\mathcal{F}$}\}\,.\]
Since every place $Q\in \mathbb{P}(F_i)$ lying above a place in $Split(\mathcal{F}/F_0)\cup Cram(\mathcal{F}/F_0)$ is a rational place (i.e. of degree one), we have that
\begin{equation}\label{splitcram}
N(F_i)\geq [F_i:F_0]|Split(\mathcal{F}/F_0)|+|Cram(\mathcal{F}/F_0)|\,,
\end{equation} where $N(F_i)$ is the number of rational places of $F_i$.

Notice that if for an admissible  sequence of function fields  $\mathcal{F}=(F_0, F_1, \ldots)$ we have that $Split(\mathcal{F}/F_0)\neq \emptyset$ (in other words, there is a rational place  $P$ in $F_0$ that splits completely in each extension $F_i/F_0$) then, by Hurwitz genus formula (\cite[Theorem 3.4.13]{Stichbook09}), we have that $g(F_i)\rightarrow\infty$ as  $i\rightarrow\infty$  so that $\mathcal{F}$ is actually a tower. 
%
%

  Given a finite extension $E/F$ and a place $P\in \P(F)$ there are finitely many places $Q\in \P(E)$ lying above $P$. We will write $Q|P$ when $Q$ lies over $P$. The extension $E/F$ is said to be \emph{tame} if the ramification index $e(Q|P)$ is relatively prime to the characteristic of $\F_q$, for all places $P\in \P(F)$ and all $Q|P$. We shall say that an admissible sequence $\m{F}=(F_0, F_1, \ldots)$ of function fields over $\F_q$ is \emph{tame} if all the extensions $F_i/F_0$ are tame.

\section{Kummer Type Towers}

  As we said in the Introduction,  it is well-known that a tame recursive tower is asymptotically good if it has non-empty splitting locus and finite ramification locus. The next result, proved in \cite{Chara2011}, gives sufficient conditions in order to have non-empty splitting locus, in the particular class of sequences of Kummer type recursively defined by \eqref{ecuintro}.

We will use the following notation. For a given rational function $f\in \F_q(T)$ the set of zeros of $f$ in an algebraic closure $\overline{\F}_q$ of $\F_q$ will be denoted either by $Z_f$, or by $Z_{f(T)}$ in case we need to specify $f(T)$.
\begin{prop}\label{phitheorem}
Let $m>r\geq 1$ be such that $\gcd(m,m-r)=1$ and $\gcd(m,q)=1$. Let $\alpha \in \F_q^\ast$ and consider de rational functions $$a(T)=T^m \qquad \text{and} \qquad b(T)=\frac{T^m-\alpha f(T)+\alpha}{f(T)},$$ where $f(T)\in \F_q[T]$ is a polynomial of degree $r$. If $\F_q$ is a splitting field for $T^m+\alpha$ and $Z_f\cap Z_{T^m+\alpha}= \emptyset$, then for the $(a,b)$-recursive sequence of function fields $\mathcal{F}=(F_0, F_1, \ldots)$ we have that
\[|Split(\mathcal{F}/F_0)|\geq m\,,\] and $$N(F_i)\geq m^{i+1}+1.$$
\end{prop}

Since we have non-empty splitting locus for the Kummer sequences recursively defined by \eqref{ecuintro}, we shall focus, from now on, in finding sufficient conditions in order to ensure finite ramification. First we give a simple result which will be useful later. We will denote by $x(P)$ the residue class mod $P$ of $x \in F$.
%

\begin{lem}\label{lemaprevio}
  Let $\m{F}=(F_0, F_1, \ldots)$ be an admisible recursive sequence of function fields over $\F_q$ defined by the equation $H(S,T)=0$, where $H\in \F_q[S,T]$. Assume that there is a set $S_0\subset \overline{\F}_q$ such that if $\gamma \in S_0$ and $H(\beta, \gamma)=0$ then $\beta \in S_0$. Let $\{x_i\}_{i\geq 0}$ be a sequence of trascendental elements over $\F_q$ such that $F_0=\F_q(x_0)$ and $F_{i+1}=F_i(x_{i+1})$ where $H(x_i, x_{i+1})=0$ for all $i\geq 0$. Let $Q$ be a place in $\P(F_i)$ such that $x_i(Q)\in S_0$. Then $x_0(Q)\in S_0$.
\end{lem}
 \begin{proof}
 Let   $Q$ be a place of  $F_i$ such that $x_i(Q)\in S_0$. Since $\m{F}$ is recursively defined by $H$, we have that $H(x_{i-1},x_i)=0$ for all $i\geq 0$. By reducing this equation modulo $Q$ we obtain $H(x_{i-1}(Q),x_i(Q))=0$, and by hypothesis $x_{i-1}(Q) \in S_0$. Now, since $H(x_{i-2},x_{i-1})=0$, the reduction modulo $Q$ of this equation shows that $x_{i-2}(Q)\in S_0$. Continuing in this way, we arrive to the desired conclusion.
 \end{proof}

Next we prove a proposition giving sufficient conditions for the finiteness of the ramification locus of a particular class of recursive sequences of Kummer type. Recall that if $g(T)\in \F_q[T]$, we denote by $Z_{g}$ the set of zeros of $g(T)$ in an algebraic closure $\overline{\F}_q$ of $\F_q$.

\begin{prop}\label{finiteram}
  Let $m\geq 2$ be an integer with $q\equiv 1 \mod m$. Consider the sequence $\m{F}=(F_0, F_1, \ldots)$ of function fields over $\F_q$ defined recursively by the equation \begin{equation}\label{ecufinitefam}
    y^m=\frac{b_1(x)}{b_2(x)}
  \end{equation} where $b_1(T), b_2(T)\in \F_q[T]$ are coprime polynomials such that $\deg(b_1(T))=m$ and $\deg(b_2(T))=m-r$ with $\gcd(m,r)=1$. Then $\m{F}$ is a tame admissible sequence. Assume now that there is a finite set $S_0\subset \overline{\F}_q$ with the following properties:
  \begin{enumerate}[1.]
    \item\label{itemifiniteram} $Z_{b_1} \subset S_0.$
    \item\label{itemiifiniteram} $Z_{b_2} \subset S_0$.
    \item\label{itemiiifiniteram} $Z_{\sigma_\gamma}\subset S_0$, for all $\gamma \in S_0$, where $\sigma_\gamma(T)=b_2(T)\gamma^m-b_1(T)$.
  \end{enumerate}
    Then $Ram(\m{F}/F_0)$ is a finite set. More precisely, if $P\in \P(\F_0)$ is a ramified place in the sequence $\m{F}$ then either $P=P_\infty$ is the pole of $x_0$ in $F_0$, or $P$ is the zero of $x_0-\gamma$, for some $\gamma \in S_0$.
\end{prop}
\begin{proof}
By hypothesis each extension $F_n/F_{n-1}$ is cyclic of degree $m$. It is also easy to see that the pole $P_\infty$ of $x_0$ in $F_0$ is totally ramified in the sequence $\m{F}$. Therefore $\F_q$ is the full constant field of each $F_n$, and then $\m{F}$ is admissible and tame.

  Suppose that $P\in \P(F_0)$ is ramified in $F_n/F_0$. Choose $Q\in \P(F_n)$ above $P$ such that $e(Q|P)>1$ and let $P_i=Q\cap F_i$ be the restriction of $Q$ to $F_i$, for each $i=0, 1, \ldots, n$. Since $Q|P$ is ramified, then $P_{i+1}|P_i$ is ramified for some index $i$.

  From the defining equation, $$x_{i+1}^m=\frac{b_1(x_i)}{b_2(x_i)},$$ and from the ramification theory of Kummer extensions (see, for example, \cite[Proposition 3.7.3]{Stichbook09}), it follows that $P_{i+1}$ is either a zero or a pole of $x_{i+1}$ in $F_{i+1}$.

  If $P_{i+1}$ is a zero of $x_{i+1}$, we have that $x_{i+1}(Q)=0$.
   By reducing the equation $x_{i+1}^mb_2(x_i)=b_1(x_i),$ modulo $Q$ we obtain $$0=x_{i+1}(Q)^mb_2(x_i(Q))=b_1(x_i(Q)),$$ and this implies that $x_i(Q)=\gamma$ for some $\gamma\in \overline{\F}_q$ such that $b_1(\gamma)=0$. Thus $x_i(Q)\in S_0$ by \eqref{itemifiniteram}.

  Suppose now that $P_{i+1}$ is a pole of $x_{i+1}$. Then $x_{i+1}^{-1}\in P_{i+1} \subset Q$. Hence $x_{i+1}^{-1}(Q)=0$. Since $b_2(x_{i})=b_1(x_i)x_{i+1}^{-m}$, by reducing modulo $Q$ we have that $b_2(x_{i}(Q))=b_1(x_i(Q))(x_{i+1}^{-1}(Q))^{m}=0$ and this implies that that $x_i(Q)=\gamma$ for some $\gamma\in \overline{\F}_q$ such that $b_2(\gamma)=0$. Therefore $x_i(Q)\in S_0$ by \eqref{itemiifiniteram}. By \eqref{itemiiifiniteram} and Lemma~\ref{lemaprevio} we have that $x_0(Q)\in S_0$.

Now we can easily see that if $P\in \P(F_0)$ is a ramified place in $\m{F}$ then, $P=P_\infty$ if $v_{P_1}(x_1)<0$ and $P$ is the zero of $x_0-\gamma$, for some $\gamma \in S_0$, if $v_{P_1}(x_1)\geq 0$. Hence $Ram(\m{F}/F_0) \subset\{P_{x_0-\gamma}\,:\,\gamma \in S_0\}\cup \{P_\infty\}$ and since $S_0$ is finite, $\m{F}$ has finite ramification locus.

  \end{proof}

%

Now we can prove one of our main results.

\begin{thm}\label{teofiniteram}
  Let $m\geq 2$ be an integer and $q$ a prime power such that $q\equiv 1 \mod m$. Let $\alpha \in \F_q$ such that $T^m+\alpha$ splits into linear factors in $\F_q$ and let $f(T)\in \F_q[T]$ be a separable polynomial of degree $m-r$ with $\gcd(m,r)=1$ such that $Z_{T^m+\alpha}\cap Z_{f}=\emptyset$.  Assume that there is a finite set $S_0\subset \overline{\F}_q$ with the following properties:
  \begin{enumerate}[1.]
    \item\label{itemifiniteram} $Z_{T^m-\alpha f(T)+\alpha}\subset S_0.$
    \item\label{itemiifiniteram} $Z_{f} \subset S_0$.
    \item\label{itemiiifiniteram} $Z_{\sigma_\gamma}\subset S_0$, for all $\gamma \in S_0$, where $\sigma_\gamma(T)=f(T)(\gamma^m+\alpha)-(T^m+\alpha)\in \overline{\F}_q[T]$.
  \end{enumerate} Then the sequence $\m{F}=(F_0, F_1, \ldots)$ of function fields defined by the equation \begin{equation}\label{ecufinitefam}
    y^m=\frac{x^m - \alpha f(x)+\alpha}{f(x)}
  \end{equation}
  is an asymptotically good tower of Kummer type over $\F_q$ and
$$\lambda(\m{F})\geq\frac{2m}{|S_0|-1}>0.$$
\end{thm}

\begin{proof}
As in Proposition~\ref{finiteram}, we have by hypothesis that each extension $F_n/F_{n-1}$ is cyclic of degree $m$, and the pole $P_\infty$ of $x_0$ in $F_0$ is totally ramified in the sequence $\m{F}$. Therefore $\F_q$ is the full constant field of each $F_n$, and then $\m{F}$ is admissible and tame.

  Using Proposition~\ref{phitheorem} we have that $|Split(\m{F}/F_0)|\geq m$ and then $\m{F}$ is a tower over $\F_q$.

The fact that $S_0\subset  \overline{\F}_q$ is finite, implies that for some integer $s$ we have that $S_0 \subset \F_{q^s}$ and satisfies the conditions in Proposition~\ref{finiteram}. Therefore we have that if $P\in \P(\F_0)$ is a ramified place in the tower $\m{F}$, then $P=P_\infty$ or $P$ is the zero of $x_0-\gamma$, for some $\gamma \in S_0$. Then the ramification locus is finite and thus the tower has a finite genus over $\F_{q^s}$. Since the genus of a tower does not change in constant field extensions, we conclude that $\m{F}$ has finite genus over $\F_q$.

Finally, since $\m{F}$ is a tame recursive tower with non-empty splitting locus and finite ramification locus, Theorem 2.1 of \cite{GST97} implies that $\m{F}$ is an asymptotically good tower of Kummer type over $\F_q$. More over, since $|Split(\m{F}/F_0)|\geq m$ and $Ram(\m{F}/F_0)\subset S_0\cup \{\infty\}$ then $$\lambda(\m{F})\geq\frac{2m}{|S_0|+1-2},$$ as desired.
  \end{proof}

  \begin{ex}\label{ejemtorreoptimaf9}
    Let $m=2$, $q=9$. 
    Let $\m{G}=(G_0, G_1, \ldots)$ be defined recursively by $$y^2=\frac{x^2-x+1}{x}.$$ This equation is of the form \eqref{ecuintro} with $f(x)=x$ and $\alpha=1$.
    In this case, $T^2+1$ has $2$ simple roots in $\F_9$ and $f(T)$ is a separable polynomial of degree $1$ with no common roots with $T^2+1$. Then $$|Split(\m{G}/G_0)|\geq 2,$$ by Proposition~\ref{phitheorem}. The set $S_0=\F_3 \subset \F_9$ satisfies conditions \eqref{itemifiniteram} to \eqref{itemiiifiniteram} of Theorem~\ref{teofiniteram}. Hence $$Ram(\m{G}/G_0) \subseteq \{P_\infty, P_{x_0}, P_{x_0-1}, P_{x_0-2}\},$$ then $\m{G}$ is an asymptotically good tower of Kummer type over $\F_9$ with
    $$\lambda(\m{G})\geq 2.$$
    Since $$2\geq A(9)\geq \lambda(\m{G})\geq 2,$$ we see that this tower over $\F_9$ is asymptotically optimal, i.e., $\lambda(\m{G})=A(9)$.
  \end{ex}

  \begin{rem}
    Note that the tower $\m{G}$ in the previous example has, in fact, finite ramification locus over $\F_3$. However, Theorem~\ref{teofiniteram} only allow us to say that $\m{F}$ has positive splitting over $\F_9$. Notice that $\m{G}$ can be described also by $$y^2=\frac{(x+1)^2}{4x}.$$ By \cite[Remark 5.9]{GSR03} we have that $\m{G}$ is a subtower of $$y^2=\frac{x^2+1}{2x},$$ which is optimal over $\F_{p^2}$.
  \end{rem}

   \begin{ex}\label{ejemtorrebuenaf9}
    Let $m=2$, $q=9$. Let $\m{H}=(H_0, H_1, \ldots)$ be defined recursively by $$y^2=\frac{x(x-1)}{x+1}.$$ This equation is of the form \eqref{ecuintro} with $f(x)=x+1$ and $\alpha=1$.
    Again in this case we have that $|Split(\m{H}/H_0)|\geq 2.$ The finite field $\F_9$ can be represented as $\F_9=\F_3(\delta)$ with $\delta^2+2\delta +2=0$. The set $S_0=\{0,1,2,\delta, \delta^3, \delta^5, \delta^7\}\subset \F_9$ satisfies conditions \eqref{itemifiniteram} to \eqref{itemiiifiniteram} of Theorem~\ref{teofiniteram}, and then $\m{H}$ is an asymptotically good Kummer type tower over $\F_9$ with
    $$\lambda(\m{H})\geq \frac 2 3.$$
  \end{ex}

\begin{rem}
    Notice that the tower $\m{H}$ in the previous example can be described also by $$y^2=\frac{x(x+2)}{x+1}.$$ Using this equation in \cite[Example 4.3]{GSR03} the authors proved that $\m{H}$ is an asymptotically good tower and the same bound for its limit was obtained.
  \end{rem}

  Theorem~\ref{teofiniteram} is stated for towers whose defining equations have coefficients in any finite field. Hence we can find examples of towers  whose defining equations have  coefficients in non-prime fields. In fact, if we perform a computer search for all possible equations of the type \eqref{ecufinitefam} satisfying the conditions of Theorem~\ref{teofiniteram} for $q = 9$ we obtain a long list of equations, and therefore of towers, which at first glance seem totally different from each other. In particular, we obtain some equation whose coefficients are purely in $\F_3$ while the vast majority has coefficients in $\F_9$. However, by a suitable change of variables, it can be shown that all of them are either equivalent to the tower in Example~\ref{ejemtorreoptimaf9} or to the tower in Example~\ref{ejemtorrebuenaf9}. That is, the towers in the above examples are the only two towers with defining equations of the type \eqref{ecufinitefam} and satisfying the conditions of Theorem~\ref{teofiniteram} with a finite set $S_0\subset \F_9$.

For example, other equations defining the asymptotically optimal Kummer tower of Example~\ref{ejemtorreoptimaf9} are given in Table~\ref{tablatorreoptima}.

%

  \begin{table}[ht]
\caption{Other examples of equations defining the same tower as in Example~\ref{ejemtorreoptimaf9}.\label{tablatorreoptima}}
{\begin{tabular}{cccc} 
\hline
$\alpha$ & $f(T)$ & defining equation & change of variables \\
 \hline
     &&&\\
    $\delta +1$ & $(\delta +2)T$ & $y^2=\frac{x^2-(\delta+1)(\delta +2)x+\delta +1}{(\delta +2) x}$ & $X=\delta x$; $Y=\delta y$ \\
    &&&\\
        $\delta +1$ & $(2\delta +1)T$ & $y^2=\frac{x^2-(\delta+1)(2\delta +1)x+\delta +1}{(2\delta +1) x}$ & $X=\delta^4 x$; $Y=\delta^4 y$ \\
        &&&\\
            $2$ & $(\delta +1)T$ & $y^2=\frac{x^2-2(\delta+1)x+2}{(\delta +1) x}$ & $X=\delta^3 x$; $Y=\delta^3 y$ \\
            &&&\\
                $2\delta +2$ & $2\delta T$ & $y^2=\frac{x^2-(2\delta +2)2\delta x+2\delta +2}{2\delta x}$ & $X=2x$; $Y=2y$ \\
                &&&\\
                $1$ & $2T$ & $y^2=\frac{x^2-2x+1}{2x}$ & $X=\delta^5 x$; $Y=\delta^5 y$ \\
                &&&\\
                 $2\delta +2$ & $\delta T$ & $y^2=\frac{x^2-(2\delta+2)\delta x+2\delta +2}{\delta x}$ & $X=\delta^2 x$; $Y=\delta^2 y$ \\
                  &&&\\
                  \hline
\end{tabular}}
\end{table}

  Given that all the equations in Table~\ref{tablatorreoptima} define the same tower and satisfy the conditions of Theorem~\ref{teofiniteram} we wonder in which cases different equations  satisfying these conditions will give us the same tower.   As a response to this question we have the following proposition.
  \begin{prop}\label{propomismas}
    Let $\alpha \in \F_q^*$ and $f(T)\in \F_q[T]$ be such that the equation \begin{equation}\label{ecurem1}
    y^m=\frac{x^m-\alpha f(x)+\alpha}{f(x)}
    \end{equation} defines a tower which satisfies the conditions of Theorem~\ref{teofiniteram}. Then, if for $c\in \F_q^*$ we consider $\beta=c^{-m}\alpha\in \F_q^*$ and $g(T)=f(cT)\in \F_q[T]$, we have that the equation $$y^m=\frac{x^m-\beta \,g(x)+\beta}{g(x)}$$ defines the same tower as the one defined by \eqref{ecurem1} and also satisfies Theorem~\ref{teofiniteram}.
  \end{prop}
\begin{proof}
    By applying to \eqref{ecurem1} the change of variables $x=cX$, $y=cY$ we get
    $$c^mY^m=\frac{c^mX^m-\alpha f(cX)-\alpha}{f(cX)}.$$ Thus $$Y^m=\frac{X^m-c^{-m}\alpha\, g(X)+c^{-m}\alpha}{g(X)}=\frac{X^m-\beta\, g(X)+\beta}{g(X)}$$ defines the same tower.

    Moreover, since $Z_{T^m+\alpha} \cap Z_f =\emptyset$ then $Z_{T^m+\beta} \cap Z_g =\emptyset$. Otherwise, if $x \in Z_{T^m+\beta} \cap Z_g$ then $x^m+\beta=0$ and $g(x)=0$. This means that $x^m+c^{-m}\alpha=0$ and $f(cx)=0$, and since $c\in \F_q^*$ we get $(cx)^m + \alpha =0$. But in this case $cx \in Z_{T^m+\alpha} \cap Z_f$ which is a contradiction.

    Let $S_0$ be the set of Theorem~\ref{teofiniteram} for $f$. We define $S_0^g=\{c^{-1}\lambda:\lambda \in S_0\} \subset \F_q$. Then $S_0^g$ satisfies:
    \begin{enumerate}[1.]
      \item $\quad$

      \vspace*{-1.18 cm} \begin{align*}
        Z_{T^m-\beta\,g(T)+\beta}&=\{x: x^m-\beta\,g(x)+\beta=0\}\\
         &= \{x:x^m-c^{-m}\alpha f(cx)+c^{-m}\alpha=0\} \\
        &= \{c^{-1}(cx):(cx)^m-\alpha f(cx)+\alpha=0\} \\
        &\subset \{c^{-1}\lambda: \lambda \in S_0\}=S_0^g.
      \end{align*}
      \item $Z_g=\{x: g(x)=0\}=\{c^{-1}(cx): f(cx)=0\} \subset \{c^{-1}\lambda: \lambda \in S_0\}=S_0^g$.
      \item For all $\gamma \in S_0^g$ we have that $\gamma=c^{-1}\delta$ with $\delta \in S_0$ and \begin{align*}
        Z_{\sigma_{\gamma}}&=\{x:g(x)(\gamma^m+\beta)-(x^m+\beta)=0\}\\
        &= \{x:f(cx)((c^{-1}\delta)^m+\alpha)-((cx)^m+\alpha)=0\} \\
        &= \{c^{-1}\lambda : f(\lambda)(\delta^m+\alpha)-(\lambda^m+\alpha)=0\}\\
        &= \{c^{-1}\lambda : \sigma_{\delta}(\lambda)=0\}\\
        &\subset \{c^{-1}\lambda: \lambda \in S_0\} =S_0^g.
      \end{align*}
    \end{enumerate}
    Therefore, for each element in $\F_q^\ast$ we have an equation which defines the same tower as \eqref{ecurem1} and satisfies the conditions of Theorem~\ref{teofiniteram}.
\end{proof}

  For those cases where $c^m=1$, we have the following direct consequence of the above proposition.

  \begin{cor}\label{coropropomismas}
     For each $m$-th root $c$ of $1$ in $\F_q$, we have that the equation $$y^m=\frac{x^m-\alpha \,g(x)+\alpha}{g(x)},$$ with $g(T)=f(cT)\in \F_q[T]$, defines the same tower as \eqref{ecurem1} and satisfies the conditions of Theorem~\ref{teofiniteram}.
  \end{cor}

   Proposition~\ref{propomismas} has an important computacional consequence. Namely when making a computer search for all possible equations that define towers satisfying the conditions of  Theorem~\ref{teofiniteram} over $\F_q$, we will actually find $q-1$ equations representing the same tower. Moreover, the above corollary tell us that for every $\alpha$ there are as many equations that define the same tower as $m$-th roots of $1$ in $\F_q$.

  Let us now look at some other examples of towers whose defining equations have coefficients in non-prime fields. We consider first the case $m=2$ and $q=25$.

  \begin{ex}~\label{ejem25}
    Let us represent the finite field $\F_{25}$ as $\F_5(\delta)$ with $\delta^2+4\delta +2=0$. Consider the sequence $\m{K}=(K_0, K_1, \ldots)$ of function fields over $\F_{25}$ defined recursively by the equation $$y^2=\frac{x^2-(\delta+2) x}{(\delta +2) x+1}.$$ We have that $\F_{25}$ is a splitting field for $T^2+4$ and it it easy to check that $S_0=\{0,2\delta+4, 4\delta+3, \delta+2, 3\delta+1\}$ satisfies the conditions of Theorem~\ref{teofiniteram}. Then $\m{K}$ is a tame Kummer type tower over $\F_{25}$ with $$|Split(\m{K}/K_0)|\geq 2,$$ and $$|Ram(\m{K}/K_0)|\leq 5.$$
    Therefore, by Theorem~\ref{teofiniteram} we have that $$\lambda(\m{K})\geq \frac{2\cdot 2}{5-1}=1.$$
  \end{ex}

\begin{rem}
By using a suitable change of variables it can be shown that the tower $\m{K}$ can also by defined by the equation $$y^2=\frac{x(x+2)}{x+1},$$ which was studied by A. Garcia, H. Stichtenoth and H. R{\"u}ck in \cite{GSR03}, where it is also shown that its limit is at least 1.
\end{rem}

\begin{rem}
  Again in this case, making a computer search for all posible equations over $\F_{25}$ defining towers satisfying Theorem~\ref{teofiniteram}, we find $24$ different equations, but all of them represent the tower $\m{K}$ of Example~\ref{ejem25}. There is no other tower of this type with a finite set $S_0\subset \F_{25}$.
\end{rem}

Now we show new examples of asymptotically good Kummer type towers over $\F_9$.
  \begin{ex}\label{ejem81}
    Let us represent the finite field $\F_{81}$ as $\F_3(\delta)$ with $\delta^4+2\delta^3 +2=0$. When looking for all possible equations $$y^2=\frac{x^2-\alpha f(x) +\alpha}{f(x)},$$ defining towers of function fields over $\F_{81}$  with, for example, $\alpha=2\delta^3+2\delta^2+1$, we arrive to $8$ different possible candidates for $f(T)$. But since $\F_{81}$ has two $2-$th roots of unity, Corollary~\ref{coropropomismas} tell us that only $4$ of these equations represent different towers. Two of them are the towers of Examples \ref{ejemtorreoptimaf9} and \ref{ejemtorrebuenaf9}, and we find two more new towers:
     $$\m{I}=(I_0, I_1, \ldots)\qquad \text{with }\qquad f(T)=(2\delta^3+2\delta^2+2)T+(\delta^3+\delta^2+2),$$ and $$\m{J}=(J_0, J_1, \ldots)\qquad \text{with }\qquad f(T)=(\delta^3+\delta^2)T+(2\delta^3+2\delta^2).$$
      In both cases we find a finite set $S_0$ with $9$ elements and by Theorem~\ref{teofiniteram} we have that $$\lambda(\m{I})\geq \frac{2\cdot 2}{9-1}=\frac 1 2 \quad \text{ and } \quad \lambda(\m{J})\geq \frac{2\cdot 2}{9-1}=\frac 1 2 .$$
  \end{ex}

  \begin{rem}
 Again as before, if we look (computationally) for all possible equations of the type \eqref{ecufinitefam} satisfying the conditions of Theorem~\ref{teofiniteram} for $q = 81$ we obtain a long list of candidates. Interestingly in this case, there are no other equations representing the towers $\m{I}$ or $\m{J}$ with coefficients in $\F_3$. Moreover, it is easy to check that the coefficients in both equations are actually in $\F_9$. However the corresponding sets $S_0$ are in $\F_{81}$ and not in $\F_9$. Since the genus of a tower does not change in constant field extensions and recalling that the towers $\m{I}$ and $\m{J}$ both have non-empty splitting locus in $\F_9$, we see that in fact, they are asymptotically good towers over $\F_{9}$, each one with limit at least $1/2$. From this and the list of asymptotically good tame towers over $\F_9$ given in \cite{MaWu05}, we can say that $\m{I}$ and $\m{J}$ are new examples.
  \end{rem}

As we mentioned in the Introduction, in this paper we have worked with towers defined recursively by equation of the form \eqref{ecuintro} because they have non empty splitting locus under the conditions of Proposition \ref{phitheorem} which are easy to check. Another equation in which is already known that the splitting locus is non-empty is \begin{equation}\label{ecuintro2}
y^m=x^{m-r} f(x),
\end{equation} where $f \in \F_q[x]$ is a suitable polynomial of degree $r$ with $f(0)\neq 0$ and $\gcd(m,r)=1$. In \cite{GST97}, Garcia, Stichtenoth and Thomas studied towers defined recursively by \eqref{ecuintro2}, giving conditions in order to have finite ramification locus. Interestingly, and somehow surprisingly, when performing a computational search for this type of equations, the only examples that appeared are the so-called Fermat type towers (see \cite{GSR03}). So we are tempted to conjecture that these are the only ones of the form \eqref{ecuintro2} which are asymptotically good. Recall that Lenstra \cite{Lenstra02} proved that over a prime fields, for equations of the form \eqref{ecuintro2} there is not a finite set $S_0 \subset \overline{\F}_p$ containing the ramifications locus of the tower.

We end with the following observation. Making the change of variables $X=1/x$ and $Y=1/y$ in \eqref{ecuintro2}, we obtain the equation $$y^m=\frac{x^m}{h(x)},$$ with $h \in \F_q[x]$. In particular, for $q=9$ and $h(x)=x-1$ we have a tower recursively defined by $$y^2=\frac{x^2}{x-1},$$ which is asymptotically optimal (see \cite[Example 14.9]{GSR03}). However this example is not new as claimed in \cite{GSR03}. It is, in fact, a Fermat type tower of the form \eqref{ecuintro2} given by $$y^2=x(x-1),$$ over $\F_9$. 

\end{document}